\documentclass[11pt]{article}
\usepackage{amsmath,amsfonts}
\usepackage{amstext,amsbsy}
\usepackage{epsfig,multicol}
\usepackage{amsthm}
\usepackage{amssymb,latexsym}
\usepackage{color}
\usepackage{MnSymbol}
\usepackage{caption}

\usepackage{tikz}
\usepackage{enumerate}

\usetikzlibrary{arrows}
\usetikzlibrary{decorations.markings}

\usepackage[colorlinks=true,citecolor=black,linkcolor=black,urlcolor=blue]{hyperref}

\theoremstyle{plain}

\newtheorem{theorem}{Theorem}
\newtheorem{lemma}[theorem]{Lemma}
\newtheorem{prop}[theorem]{Proposition}

\theoremstyle{definition} 
\newtheorem{definition}{Definition}

\theoremstyle{definition}
\newtheorem*{remark}{Remark}
\theoremstyle{definition}
\newtheorem{example}{Example}
\theoremstyle{definition}

\usepackage{fullpage}

\usepackage{graphicx} 

\usepackage{tikz}

\usetikzlibrary{cd}

\usepackage[parfill]{parskip} 

\usepackage{booktabs} 
\usepackage{array} 
\usepackage{paralist} 
\usepackage{verbatim} 
\usepackage{subfig} 
\usepackage{multirow}
\usepackage{hyperref}

\def\zz{\mathcal{Z}}

\def\gon{Gon\v{c}arov }

\usepackage[nottoc,notlof,notlot]{tocbibind} 
\usepackage[titles,subfigure]{tocloft} 

\usepackage{lipsum}

\newcommand{\vanish}[1]{}

\title{\gon Polynomials in Partition Lattices  \\ and Exponential Families}
\author{ 
Dedicated to Joseph Kung  \vspace{1cm} \\  
Ayomikun Adeniran\thanks{Corresponding author: ayoijeng@math.tamu.edu} \ and \   Catherine Yan\thanks{cyan@math.tamu.edu}\\
Department of Mathematics, Texas A\&M University, College Station, TX 77843}

\date{} 

\begin{document}
\maketitle

\abstract{
Classical \gon polynomials arose in 
numerical analysis as a basis for the 
solutions of the \gon interpolation problem. These polynomials   provide a natural algebraic tool in the enumerative theory of   parking 
functions. 
By replacing  the differentiation operator with a delta operator and using the theory of finite operator calculus, 
Lorentz, Tringali and Yan introduced the sequence of generalized Gon\v{c}arov polynomials associated to a 
pair $(\Delta, \zz)$ of a delta operator $\Delta$ and an interpolation
grid $\zz$. Generalized \gon polynomials share many nice algebraic  properties and have a connection with  the theories of binomial enumeration and order statistics.  In this paper we give a complete combinatorial interpretation for any sequence of  generalized \gon polynomials. First we show that they can be realized  as weight enumerators in partition lattices. Then we give a more concrete realization in exponential families and show that these polynomials enumerate various enriched structures  of vector parking functions.}  \\



\section{Introduction}
The classical \gon interpolation problem in numerical analysis was introduced by \gon \cite{Vgonc1,Vgonc2} and Whittaker\cite{whittaker}. It asks for a polynomial $f(x)$ of degree $n$ such that the $i$th derivative of $f(x)$ at a given point $a_i$ has value $b_i$ for $i=0,1,2,...,n$. The solution is obtained by taking linear combinations of the (classical) \gon polynomials,  or the Abel-\gon polynomials, which have been studied extensively by analysts; see e.g. \cite{Vgonc1, levi, frank, haslinger}. 
\gon polynomials also play a crucial role in Combinatorics due to their close relations to parking functions. A (classical) parking function is a sequence $(a_1,a_2,...,a_n)$ of positive integers such that for every $i=1,2,...,n$, there are at least $i$ terms that are less than or equal to $i$. For example, the sequences $(1,2,3,4)$ and $(2,1,4,1)$ are both  parking functions while $(2,2,3,4)$ is not. 
The set of parking functions stays in the center  of enumerative combinatorics, with many generalizations and connections to other research areas, such as hashing and linear probing in computer science, graph theory, interpolation theory, diagonal harmonics, representation theory,  and cellular automaton. 
See the comprehensive survey \cite{yan} for more on the combinatorial theory of parking functions.

 The connection between \gon polynomials and 
combinatorics was first found by Joseph Kung, who in 
 a short note \cite{kung81} of 1981 
 proved that  classical \gon polynomials give the probability distribution of the order statistics of $n$ independent uniform random variables, and its  difference analog describes the order statistics of  discrete, injective functions. 
 These results were further developed in \cite{kung} to an explicit correspondence between classical \gon polynomials and vector parking functions. 
 Inspired by the rich theory on delta operators and finite operator calculus, which is a unified theory on linear 
operators analogous to the differentiation operator $D$ and special polynomials, 
Lorentz, Tringali, and the second author of the present paper introduced the generalized \gon polynomials \cite{goncarovdelta}  as a basis for the solutions to the \gon interpolation problem with respect to a delta operator. 
 Many algebraic and analytic properties of  classical \gon polynomials have been extended to the generalized version.

A natural question is to find the combinatorial interpretations for the generalized \gon polynomials. 
To answer this question we need to understand the combinatorial  significance of delta operators. In the third paper of the
seminal series \emph{On the Foundations of Combinatorial Theory III},  Mullin and Rota \cite{mullinrota}  developed the basic theory of delta operators and their associated  sequence of polynomials. Such sequences of polynomials are of binomial type and occur in many combinatorial problems when objects can be pieced together out of small, connected objects. Mullin and Rota's work provides a realization of binomial sequences in combinatorial problems. However, this realization is only valid for binomial sequences whose coefficients are non-negative integers, and so excludes 
many basic counting polynomials,  for example, the falling factorial $x_{(n)}=x(x-1)\cdots (x-n+1)$. Mullin and Rota hint at a generalization of their theory to incorporate such cases. 
Using the language of partitions, partition types and partition categories, Ray \cite{ray} proved that every polynomial sequence of binomial type can be realized as a weighted enumerator in  partition lattices.

In this paper we give a complete combinatorial interpretation of the generalized  \gon polynomials, first in Ray's partition lattices and then in a more concrete model, the exponential families, as described by Wilf \cite{wilf}. 
Basically,  to any polynomial sequence of binomial type 
and any given interpolation grid $\zz$
there is an associated sequence of \gon polynomials. 
While the sequence of binomial type can be realized as weighted enumerators in  partition lattices or in an exponential family,  the associated \gon polynomials count those structures which also encode vector parking functions. In other words, the generalized \gon polynomials characterize structures in 
a binomial enumeration problem that are subject to certain order-statistic constraints. Our results cover the initial attempt in \cite{goncarovdelta} which provides a combinatorial interpretation for some families of generalized  \gon polynomials in a structure called \emph{reluctant functions}.

The rest of the paper is organized as follows. In Section 2, we recall the basic theory of delta operators and binomial enumeration, as well as the concepts of generalized \gon polynomials and vector parking functions. In Section 3, we describe the realization of generalized \gon polynomials in partition lattices and weight functions. Then, in Section 4, we study a more concrete realization of \gon polynomials as type-enumerator in exponential families. 
We end the paper in Section 5 with a few closing remarks.

\section{Background}
\subsection{Delta Operators and Binomial Enumeration}
We recall the basic theory of delta operators and their associated sequence of basic polynomials as developed by  Rota, Kahaner, and Odlyzko \cite{rota2}. 
Let $\mathbb{K}$ be a field of characteristic zero and $\mathbb{K}[x]$ the vector space of all polynomials in the variable $x$ over $\mathbb{K}$. For each $a \in \mathbb{K}$, let $E_a$ denote the shift operator $\mathbb{K}[x] \rightarrow \mathbb{K}[x] : f(x) \mapsto f(x + a)$. A linear operator $\mathbf{s} : \mathbb{K}[x] \rightarrow \mathbb{K}[x]$ is called \emph{shift-invariant} if $\mathbf{s} E_a = E_a \mathbf{s}$ for all $a \in \mathbb{K}$, where the multiplication is the composition of operators.

\begin{definition}
 A delta operator $\Delta$ is a shift-invariant operator satisfying $\Delta(x)=a$ for some nonzero constant $a$.
\end{definition}


\begin{definition}
Let $\Delta$ be a delta operator. A polynomial sequence 
$\{p_n(x)\}_{n\geq 0}$ is called the sequence of basic polynomials, or the associated basic sequence of $\Delta$ if 
\begin{enumerate}[(i)]
\item $p_0(x)=1$;
\item Degree of $p_n(x)$ is $n$ and $p_n(0)=0$ for each $n\geq 1$;  
\item $\Delta(p_n(x)) = np_{n-1}(x)$.
\end{enumerate} 
\end{definition}

Every delta operator has a unique sequence of basic polynomials, which is a sequence of binomial type (or binomial sequence) that satisfies  
\begin{equation}‎ \label{eq:binomial} 
p_{n}(u+v)= \sum\limits_{i \geq 0} \binom{n}{i}p_{i}(u) p_{n-i}(v),
\end{equation}
for all $n\geq 0$.
Conversely, every polynomial sequence  of binomial type is the associated basic sequence of some delta operator.

Let $\mathbf{s}$ be a shift-invariant operator, and $\Delta$ a delta operator. Then $\mathbf{s}$ can be expanded uniquely as a formal power series of $\Delta$.  If 
\[
\mathbf{s}= \sum_{k \geq 0}   \frac{a_k}{k!} \Delta^k ,
\]
we say that $f(t)=\sum_{k \geq 0} \frac{a_k}{k!} t^k$ is the 
$\Delta$-indicator of $\mathbf{s}$. In fact, the correspondence 
\[
f(t)=\sum_{k \geq 0} \frac{a_k}{k!} t^k
\longleftrightarrow 
\sum_{k \geq 0} \frac{a_k}{k!} \Delta^k
\]
is an isomorphism from the ring $\mathbb{K}[[t]]$ of formal power series in $t$ onto the ring of shift-invariant operators.  Under this isomorphism, a shift-invariant operator is invertible if and only if its $\Delta$-indicator $f(t)$ satisfies $f(0) \neq 0$, and it is a delta operator if and only if $f(0)=0$ and $f'(0)\neq 0$, i.e., $f(t)$ has a compositional inverse $g(t)$ satisfying $f(g(t))=g(f(t))=t$.

Another important result is the generating function for the sequence of basic polynomials $\{p_n(x)\}_{n\geq0}$ associated to a delta operator $\Delta$. Let $f(t)$ be the $D$-indicator of $\Delta$, where $D=d/dx$ is the differentiation operator. 
Let $g(t)$ be the  compositional inverse of $f(t)$. Then, 
\begin{equation}\label{eq:2}
\sum\limits_{n\geq 0} p_{n}(x) \frac{t^n}{n!}= \exp\left(xg(t)\right).
\end{equation}
The operator $\Lambda=g(D)$ is called the \emph{conjugate delta operator} of $\Delta$, and 
$\{p_n(x)\}_{n \geq 0}$ is the \emph{conjugate sequence} of 
$\Lambda$.  It is easy to see that  if $p_n(x) = \sum\limits_{k \geq 1} p_{n,k} x^k$, then 
$g(t) = \sum\limits_{k \geq 1} p_{k,1} \frac{t^k}{k!}$.

Polynomial sequences of binomial type are closely related to 
the theory of binomial enumeration. Consider the following model. 
Assume $\mathcal{B}$ is a family of discrete structures.
For a finite set $E$, let $\Pi(E)$ be the poset of all partitions $\pi$ of $E$, ordered by refinement, and write 
$|\pi|$ for the number of blocks of $\pi$. Define a 
\emph{$k$-assembly of $\mathcal{B}$-structures on $E$}  as  a partition $\pi$ of the set $E$ into $|\pi|=k$ blocks such that each block of $\pi$ is endowed with a
$\mathcal{B}$-structure. Let $B_k(E)$ denote the set of all such $k$-assemblies. 
For example, when $\mathcal{B}$ is a set of
rooted trees, a $k$-assembly of $\mathcal{B}$-structures on $E$ is a 
forest of $k$ rooted trees with vertex set $E$. 
We can also take $\mathcal{B}$ to be other structures, such as permutations, complete graphs, posets, etc. 
Assume that the cardinality of $B_k(E)$ depends only on the 
cardinality of $E$, but not its content.  In other words, there is a bijection between $B_k(E)$ and $B_k([n])$ where $[n]=\{1,2,...,n\}$ and $|E|=n$.

\vanish{
Next, we give a quick survey of the theory of binomial enumeration, based on the foundations laid in \cite{joyal} and \cite{mullinrota}.
Recall from \cite{joyal} that a \emph{species} $B$ is a covariant endofunctor on the category of finite sets and bijections. 
Given a finite set $E$, an element $s \in B(E)$ is called a $B$-structure on $E$. 
Following \cite{navarota}, let $\Pi(E)$ be the poset of all partitions $\pi$ of $E$, ordered by refinement, and write $|\pi|$ for the number of blocks of $\pi$.
A $k$-assembly of $B$-structures on $E$ is then a partition $\pi$ of the set $E$ into $|\pi|=k$ blocks such that each block of $\pi$ is endowed with a
$B$-structure. Let $B_k(E)$ denote the set of all such $k$-assemblies. For example, when $B$ is a set of
rooted trees, a $k$-assembly of $B$-structures on $E$ is a $k$-forest of rooted trees with vertex set $E$. We note that there is  a bijection between $B_k(E)$ and $B_k([n])$ where $[n]=\{1,2,...,n\}$ and $|E|=n$. We can also take $B$ to be other structures, such as permutations, complete graphs, posets, etc. The significance of this theory is that the enumeration of such structures gives rise to polynomials of binomial type.
}

\begin{definition}
Let
\[
b_{n,k}=\left\{ 
\begin{array}{ll} 
|B_k([n])|, &  \text{if } k\leq n \\
0,  &  \text{if } k>n,
\end{array}\right. 
\]
where $b_{0,0}=1$ and $b_{n,0}=0$ for $n\geq 1$.  
\end{definition}

\begin{theorem}[\cite{mullinrota}] \label{mrot}
Assume $b_{1,1} \neq 0$. 
If $b_n(x) = \sum_{k=1}^n b_{n,k}x^k$ is the enumerator for assemblies of $\mathcal{B}$-structures on $[n]$, 
then $(b_n(x))_{n\geq 0}$ is a sequence of polynomials of binomial type.
\end{theorem}

Theorem \ref{mrot} provides a realization of binomial sequences in combinatorial problems. If we think of $x$ as a positive integer  such that $|X|=x$ for some set $X$, then we can interpret $b_n(x)$ as the number of assemblies of $\mathcal{B}$-structures on $[n]$, where each block carries a label from $X$. From this viewpoint, it is easy to see that 
 $(b_n(x))_{n\geq 0}$ is of binomial type. 

This realization is only valid for binomial sequences whose coefficients are non-negative integers, and so excludes 
many polynomial sequences naturally appearing in combinatorics, 
for example, the falling factorials $x_{(n)}$.
Mullin and Rota expanded their construction slightly by
 considering the \emph{monomorphic classes},  in which different blocks receive different labels from $X$, and hence the counting polynomial becomes $\tilde b_n(x)=\sum_{k=1}^n b_{n,k} x_{(k)}$. 
 Ray \cite{ray} extended Mullin-Rota's theory and developed the concept of partition categories, and he proved that any binomial sequence can be realized as a weight enumerator in  partition lattices. We will use Ray's model in Section 3. 

\subsection{Generalized Gon\v{c}arov Polynomials and  Vector Parking Functions}

\vanish{
The application of \gon polynomials in Probability Theory and Combinatorics was first pointed out by Kung \cite{kung81} in his short note of 1981, \emph{A probabilistic interpretation of the \gon and related polynomials}. Later Kung and Yan \cite{kung} showed that \gon polynomials is the algebraic counterpart of the vector parking functions. 
Inspired by the rich theory of delta operators, Lorentz, Tringali and Yan \cite{goncarovdelta} extended the \gon interpolation problem by replacing the differential operator $D$ with a delta operator.
The solution of this problem yields the generalized \gon  polynomial. When $\Delta=D$, the classical Gon\v{c}arov polynomials are recovered. Generalized \gon  polynomials enjoy many nice algebraic properties and carry a combinatorial interpretation that combines the idea of binomial enumeration and order statistics.  } 

Let $\mathcal{Z}=(z_i)_{i\geq 0}$ be a fixed sequence with values in $\mathbb{K}$, where $\mathbb{K}$ is the scalar field.
For our purpose, it suffices to  take $\mathbb{K}$ to be $\mathbb{Q}$, 
$\mathbb{R}$, or $\mathbb{C}$. 
We call  $\mathcal{Z}$ the interpolation grid and $z_i \in \mathcal{Z}$  the $i$-th interpolation node. Let $\mathcal{T}=(t_n(x;\Delta,\zz))_{n\geq 0}$ 
be the unique sequence of polynomials that satisfies
\begin{equation}\label{eq:gp-def}
    \varepsilon_{z_i}\Delta^i(t_n(x; \Delta, \zz)) = n!\delta_{i,n},
\end{equation}
where $\varepsilon_{z_i}$ is evaluation at $z_i$. 

\begin{definition}
The polynomial sequence $\mathcal{T}=(t_n(x; \Delta, \zz))_{n\geq 0}$ determined by \eqref{eq:gp-def} is called the \textit{sequence of generalized \gon  polynomials} associated with the pair $(\Delta, \mathcal{Z})$ and $t_n(x;\Delta, \zz)$ is the $n$-th generalized \gon polynomial relative to the same pair.
\end{definition}
This sequence $\mathcal{T}$ has a number of interesting algebraic properties. One of them is a recurrence formula described as follows: Let $t_n(x)=t_n(x; \Delta, \zz)$   and $\{p_n(x)\}_{n\geq 0}$ be the basic sequence associated to $\Delta$. Then
 \begin{equation}\label{eq:gp-recurrence}
    p_n(x)= \sum_{i=0}^n \binom{n}{i}\ p_{n-i}(z_i) \  t_i(x).
\end{equation}
We remark that by definition,  to compute the generalized \gon polynomials given the basic sequence, one would find the 
conjugate operator $\Lambda$ via \eqref{eq:2}, compute 
$\Delta$ by solving for the compositional inverse of the $D$-indicator of $\Lambda$,  and then find the $n$-th polynomial 
$t_n(x)$ of the sequence
by using \eqref{eq:gp-def}. 
The computation required in this process  can be quite involved. 
However, \eqref{eq:gp-recurrence} gives a recursive formula which can be used as an alternative definition for $t_n(x)$, 
which is much more convenient in combinatorial problems. 
 For  other algebraic properties of generalized \gon polynomials, see \cite{goncarovdelta}.

Classical \gon polynomials  have a combinatorial interpretation in enumeration. 
Let $\vec{u}=(u_i)_{i\geq 1}$ be a sequence of non-decreasing positive integers. A (vector) $\vec{u}$-parking function of length $n$ is a sequence $(x_1,...,x_n)$ of positive integers whose order statistics, i.e., the nondecreasing rearrangement $(x_{(1)}, x_{(2)},..., x_{(n)})$,  satisfy the inequalities
 $x_{(i)} \leq u_i$ for all $i=1, \dots, n$.  Vector parking functions can be described via a parking 
process of $n$ cars trying to park along a line of $x \geq n$ 
parking spots. Cars enter one by one in order, and before parking, each driver has a preferred parking spot. Each driver goes to her preferred spot directly  and parks in the first spot available from there, if there exists one. A $\vec{u}$-parking function is a sequence of drivers' preferences such that at least $i$ cars prefer to  park in the first $u_i$ spots, for all $i=1, \dots, n$. When $u_i=i$, we recover the classical parking functions, which were originally introduced by Konheim and Weiss \cite{weiss} and are 
the preference sequences such that $x=n$ and every driver can 
find some spot to park.  In general, 
the $n$-th  \gon polynomials associated to the pair $(D, -\zz)$ counts the 
number of $\vec{z}$-parking functions of length $n$, where 
$\vec{z}=(z_0, z_1, \dots, z_{n-1})$ is the initial segment of the grid $\zz$; see \cite{kung}.

A concrete realization of $t_n(x;\Delta, \zz)$ for some other delta operators $\Delta$  is found in a combinatorial object called reluctant functions whose underlying structure are families of labeled trees. In \cite{goncarovdelta}, it is proved for some properly defined $\Delta$ and $\zz$, $t_n(x;\Delta, \zz)$ enumerates the number of reluctant functions in
a certain binomial class $\mathcal{B}$ whose label sequences are $\vec{z}$-parking functions.   The object of the present paper is to extend this result and prove that for any delta operator, the generalized \gon polynomials (up to a scaling) have a realization as  a weighted enumerator in partition lattices and in any exponential family.

\section{\gon Polynomials in Partition Lattices}


In this section we give a generic, or universal realization of 
generalized \gon polynomials in weighted enumeration over  partition lattices. Our result is built on Ray's solution of the realization problem for arbitrary sequences of binomial types in the context of  partition categories.  Here, we will simplify his notation and present his construction in terms of incidence algebra of partially ordered sets, which was the language originally used by combinatorialists, e.g., 
see Joni and Rota \cite{jonirota}.

For any finite set $S$, let $\Pi(S)$ denote the set of all partitions of $S$, and write $\Pi_n$ for $\Pi([n])$. Elements 
of $\Pi(S)$ are partially ordered by refinement: that is, define $\pi \leq \sigma$ if every block of $\pi$ is contained in a block of $\sigma$.  In particular, $\Pi(E)$ has a unique maximal element $\hat 1$ that has only one block and a unique minimal element $\hat 0$ for which every block is a singleton. 
Let $|\pi|$ be the number of blocks of $\pi$  and $\Pi(\pi)$ be the partitions of the set that consists of blocks of $\pi$, 
When $\pi \leq \sigma$,  the 
\emph{induced partition} $\sigma /\pi$ is the partition $\sigma$ viewed as an element of $\Pi(\pi)$.  
Define the \emph{class} of $(\pi, \sigma)$ as the sequence  
$\lambda = (\lambda_1, \lambda_2,...)$ of non-negative integers such that $\lambda_i$ is the number of blocks of size $i$ in the partition $\sigma/\pi$, for $1 \leq i \leq |\pi|$. 
It follows that 
$$
\sum_{i\geq 1} i\lambda_i = |\pi| \,\qquad \text{ and } 
\qquad \sum_{i\geq 1} \,\,\lambda_i = |\sigma|.
$$

\begin{example}
Let $E = [8]$, $\pi=\{1\},\{2\},\{345\},\{67\},\{8\}$, $\sigma=\{1345\},\{2\},\{678\} \in \Pi_8$. 
Then, $\sigma/\pi = \{(1),(345)\},\{(2)\},\{(67),(8)\} \in \Pi(\pi)$. 
 The class of $(\pi, \sigma)$ is $\lambda = (1,2,0,0,...)$, where we have  $\sum_{i\geq 1} i\lambda_i = |\pi|=5$ and $\sum_{i\geq 1} \lambda_i = |\sigma|=3$. \qed 
\end{example}

We recall the basic notation in incidence algebra. 
Let $P$ be a finite poset and $A$  a commutative ring with unity. Denote by $Int(P)$ the set of all intervals of $P$, i.e.,  the set $\{ (x,y): \ x \leq y \}$. 
The \emph{incidence algebra} $I(P, A)$ of $P$ over $A$ is the $A$-algebra of all functions 
\[
f: Int(P) \rightarrow A,  
\]
where multiplication is defined via the convolution 
\[
fg(x,y) = \sum_{ x \leq z \leq y} f(x,z) g(z,y). 
\]
The algebra $I(P, A)$ is  associative with identity $\delta$, where 
\[
\delta(x,y)=\left\{ \begin{array}{ll} 
 1, & \text{  if } x=y,  \\ 
 0, & \text{ if }  x \neq y. 
 \end{array}\right. 
 \]
An element $f \in I(P,A)$ is invertible under the multiplication if and only if $f(x,x)$ is invertible in $A$ for every $x \in P$. 

 In this paper we are concerned with the case  
$P=\Pi_n$, the  partition lattice of $[n]$,  and $A=\mathbb{K}[w_2, w_3, 
 \ldots]$, where  $w_2, w_3,...$ are  independent variables. 
 In addition, we set $w_1=1$. 

\begin{definition}
Assume  $\pi \leq \sigma$ in $\Pi_n$  and the class of $(\pi, \sigma)$ is $\lambda=(\lambda_1, \lambda_2, \ldots)$. 
Define the \emph{zeta-type function} $w(\pi, \sigma) \in I(\Pi_n,
A)$ by letting
\begin{equation} \label{Eq:zeta-type} 
w(\pi, \sigma) =  w_1^{\lambda_1} w_2^{\lambda_2} \ldots w_{|\pi|}^{\lambda_{|\pi|}}.
\end{equation} 
\end{definition} 

Note that $w(\pi, \pi)=1$ for all $\pi$. Hence $w$ is invertible. The inverse of $w$ is called the \emph{M\"obius-type function} and denoted by $\mu^w$. Explicitly, 
$\mu^w(\pi, \pi)=1$ and for $\pi < \sigma$, 
\[
\mu^w(\pi, \sigma) = - \sum_{\pi \leq \tau < \sigma} \mu^w(\pi, \tau) w(\tau, \sigma). 
\]
When all $w_i=1$, the zeta-type function and the M\"obius-type function become the zeta function and the M\"obius function of $\Pi_n$ respectively. 

\begin{example} 
Consider the lattice $\Pi_3$. Then for all $\pi < \sigma$, 
$w(\pi, \sigma)=w_2$ except that $w(\hat 0, \hat 1)=w_3$. 
Consequently, $\mu^w(\pi, \sigma)=-w_2$ if $\pi < \sigma$ except that  $
\mu^w(\hat 0, \hat 1)= 3w_2^2-w_3$. 
\qed
\end{example}

Define the \emph{zeta-type  enumerator}  $\{a_n(x;w)\}_{n \geq 0}$ and \emph{M\"obius-type enumerator} $\{b_n(x;w)\}_{n\geq 0}$ as follows. 
Let $a_0(x;w)=b_0(x;w)=1$ and for $n \geq 1$, 
\begin{eqnarray} 
a_n(x;w) & = & \sum_{ \pi \in \Pi_n} w (\hat 0, \pi)  \ x^{|\pi|}, 
\label{eq:zeta-polynomial}  \\ 
b_n(x; w) & = & \sum_{ \pi \in \Pi_n} \mu^w(\hat 0, \pi) \ x^{|\pi|}. \label{eq:mobius-polynomial} 
\end{eqnarray}

\begin{theorem} [\cite{ray}] \label{thm: Ray} 
\begin{enumerate} 
\item The polynomial sequences $\{a_n(x;w)\}_{n \geq 0}$ and 
 $\{b_n(x;w)\}_{n \geq 0}$ are of binomial type. 
\item Let $\Lambda $ be the delta operator whose $D$-indicator is given by $g(t)=t+\sum_{i \geq 2}  w_i t^i /i!$. Then 
$\{a_n(x;w)\}_{n \geq 0}$ is the conjugate sequence of $\Lambda$ and $\{b_n(x;w)\}_{n \geq 0}$ is the  basic sequence of $\Lambda$. 
\end{enumerate} 
\end{theorem}

For $n=0, 1, \dots, 4$, the polynomials $a_n(w,x)$ and $b_n(w,x)$ are 
\begin{eqnarray*} 
a_0(x;w) &=&1 ,\\ 
a_1(x;w) &= & x ,\\
a_2(x;w) & = & x^2+w_2x, \\ 
a_3 (x;w)& = & x^3 + 3w_2 x^2 + w_3 x ,\\
a_4 (x;w) & = & x^4 + 6w_2 x^3 + (4w_3+3w_2^2) x^2 + w_4x ,
\end{eqnarray*} 
and 
\begin{eqnarray*} 
b_0 (x;w)& = & 1, \\
b_1 (x;w)& = & x, \\
b_2 (x;w)& = & x^2-w_2 x, \\
b_3 (x;w)& = & x^3 -3w_2 x^2 + (3w_2^2-w_3) x, \\
b_4 (x;w)& = & x^4 -6w_2x^3 +(15w_2^2-4w_3)x^2 + (10w_2w_3-w_4-15w_2^3)x.
\end{eqnarray*} 

The linear coefficient in $b_n(w;x)$ is $\mu^w_n=\mu^w(\hat 0, \hat 1)$ in $\Pi_n$. 
Assume $\Delta$ is the conjugate delta operator of $\Lambda$. 
Then $\{a_n(x;w)\}_{n\geq 0}$ is the basic sequence of $\Delta$
and $\{b_n(x;w)\}_{n\geq 0}$ is the conjugate sequence of $\Delta$. 
The operator  $\Delta$ can be written  as 
 $\Delta= \sum_{n \geq 1} \mu^w_n D^n/n!$. 
Since $w_1=1$, each $\mu^w_n$ is a polynomial of $w_2, w_3, \dots$.  If we take $w_1$ to be a variable,  $\mu^w_n$ would be a polynomial in $w_1^{-1}, w_2,  w_3, \dots$.

The condition $w_1=1$ is equivalent to the equation $a_1(x;w)=x$. 
Since the weight variables $w_2, w_3, \ldots$ can take  arbitrary values, Theorem~\ref{thm: Ray} implies  that any polynomial  sequence $\{p_n(x)\}_{n \geq 0}$ of binomial type with $p_1(x)=x$ 
can be realized as the zeta-type weight enumerator or the 
M\"obius-type weight enumerator over partition lattices. 
Note that for any scalar $k \neq 0$, if a sequence $\{ p_n(x) \}_{n \geq 0}$ is the basic sequence of $\Delta$ and the conjugate sequence of $\Lambda$, 
then $\{ p_n/k^n\}_{ n \geq 0}$ is the basic sequence of 
$k \Delta$ and the conjugate sequence of $g(D/k)$ where 
$g(t)$ is the $D$-indicator of $\Lambda$. Hence Theorem~\ref{thm: Ray} covers all polynomial sequences of binomial type up to a scaling.

In the problem of counting  assemblies  of $\mathcal{B}$-structures outlined in Section 2.1, the enumerator $\sum_{k} b_{n,k} x^k$ in Theorem \ref{mrot} is a specialization of the 
 polynomial  $a_n(x;w)$, where $w_n$ is the number of  $\mathcal{B}$-structures on a block of size $n$.  For example, 
 when $\mathcal{B}$ is the set of rooted trees, 
 $w_n=n^{n-1}$ and hence $a_n(x;w) = x(x+n)^{n-1}$, 
 the $n$-th Abel polynomial.

Our objective is to fit the generalized \gon polynomials into this model and present a combinatorial interpretation  in terms of weight-enumeration in partition lattices. 
Following the notation of Theorem~\ref{thm: Ray}, let $\Delta$ be the conjugate delta operator of $\Lambda$. 
Given an interpolation grid $\zz$, we denote by $t_n(x; w , \zz)$ the $n$-th generalized \gon polynomial relative to the pair $(\Delta, \zz)$. We use this notation to emphasize the role of the zeta-type  function $w(\pi,\sigma)$.

To get a formula for the polynomial $t_n(x; w, \zz)$, we use the recurrence \eqref{eq:gp-recurrence} in Section 2.2. 
Since  $a_n(x;w)$ is the  basic sequence of $\Delta$, 
$\{t_n(x;w,\zz)\}_{n \geq 0}$ is the unique sequence of polynomials that satisfies the recurrence 
\begin{equation} \label{recurrdef}
a_n(x;w)= \sum_{i=0}^n \binom{n}{i} a_{n-i}(z_i;w)\ t_i(x;w,\zz).
\end{equation}
In other words,  
\begin{equation} \label{recurrdef2}
t_n(x;w,\zz)=a_n(x;w)- \sum_{i=0}^{n-1} \binom{n}{i}a_{n-i}(z_i;w)\ t_i(x;w,\zz).
\end{equation}
In particular, $t_0(x; w, \zz)=1$ and $t_1(x;w,\zz)=a_1(x;w)-a_1(z_0;w) = x-z_0$. Here we again assume $w_1=1$ and hence $a_1(x;w)=x$. Since if $\Delta$ is changed to $k \Delta$, the corresponding $t_n(x; w, \zz)$ just changes to $t_n(x;w,\zz)/k^n$, again  we  cover all the cases up to a scaling. 
 
Assume $x$ is a positive integer and $X=\{1,2,\ldots, x\}$. 
Then $a_n(x;w)$ is the zeta-type weight enumerator of all the block-labeled partitions, where each block of the partition carries a label from $X$. 
In symbols,
\[
a_n(x;w) = \sum_{ \pi \in \Pi_n } w(\hat 0, \pi) \cdot 
| \{ f: \mathrm{Block}(\pi) \rightarrow X \}|,
\]
where $\mathrm{Block}(\pi)$  is the set of blocks of $\pi$. 
For a partition $\pi$ with a block-labeling $f$, 
 we record the labeling by the list $f_{\pi}=(x_1, x_2, \dots, x_n)$, where $x_i=f(B_j)$ whenever  $i$ is in the block $B_j$ of $\pi$. 

Let $\vec{z}=(z_0, z_1, \cdots, z_{n-1})$ be the initial segment of  the grid $\zz$. Furthermore, assume that $z_0 \leq z_1 \cdots \leq z_{n-1}$ are positive integers with 
$z_{n-1} < x$. 

Define the set $\mathcal{PF}_{\pi}(\zz)$ as the set of all block-labelings of $\pi$ that are also  $\vec{z}$-parking functions, i.e., 
\begin{eqnarray} \label{def:pf_pi}
\mathcal{PF}_{\pi}(\zz) = \{ f:  \mathrm{Block}(\pi) \rightarrow X\ | \  f_\pi \text{ is a $\vec{z}$-parking function} \} .
\end{eqnarray}
More precisely, $\mathcal{PF}_{\pi}(\zz)$ is the set of block-labelings of $\pi$ such that the order statistics of 
$f_{\pi}=(x_1, x_2, \dots, x_n)$ satisfies $x_{(i)} \leq z_{i-1}$ for $i=1, \dots, n$. 
Let $PF_{\pi}(\zz)$ be the cardinality of 
$\mathcal{PF}_{\pi}(\zz)$. 

Our main result of this section is the following theorem. 
\begin{theorem} \label{thm:gp-partition-lattice} 
Assume $t_n(x;w,\zz)$ is the $n$-th generalized \gon polynomial 
defined by \eqref{recurrdef}  
 with a positive increasing integer sequence $\zz=(z_0,z_1,...)$.
Let $x$ be an integer larger than $z_{n-1}$. Then,
 \begin{equation}\label{eq:gp-lattice} 
 t_n(0;\omega,-\zz) =t_n(x;\omega,x-\zz)= \sum_{\pi \in \Pi_n} w(\hat 0,\pi)\cdot PF_{\pi}(\zz),
 \end{equation}
 where $x-\zz=(x-z_0, x-z_1, x-z_2,  \dots$) and $-\zz=(-z_0, -z_1, -z_2,\dots)$.
\end{theorem}

The first equality follows from \cite[Prop.3.5]{goncarovdelta} 
that was proved by verifying the defining equation \eqref{eq:gp-def},
and the second equality follows from  the recurrence \eqref{recurrdef} and Lemma \ref{mainlem} proved next. Note that  all three parts of \eqref{eq:gp-lattice} are polynomials of $z_0, z_1, \dots, z_{n-1}$, hence \eqref{eq:gp-lattice} is a polynomial identity.

\begin{lemma} \label{mainlem}
For every $n \geq 0$, it holds that
\begin{equation} \label{maineq}
   a_n(x;w)= \sum_{i=0}^n \binom{n}{i} a_{n-i}(x-z_i;w) 
      \sum_{\pi \in \Pi_i}  w(\hat 0,\pi)  \cdot PF_{\pi}(\zz)
\end{equation} 
\end{lemma}
\begin{proof}
Again  we assume that $x$ and $z_i$ are positive integers and $z_0 < z_1 < \cdots  < z_{n-1} < x$. 
For a finite set $E$ and $P$, 
let $\mathcal{S}(E, P)$ be the set of pairs $(\pi, f)$ where $\pi$ is a partition of the set $E$ and $f$ is a function from $\mathrm{Block}(\pi)$ to $P$. 
Then the left-hand  side of \eqref{maineq} counts the set
$\mathcal{S}([n], X)$ by the zeta-type weight function 
$w(\hat 0, \pi)$. Note that if $\pi$ has blocks $B_1, B_2, \dots, B_k$, then 
\[
w(\hat 0, \pi) = \prod_{j=1}^k w_{|B_j|}. 
\]

For a pair $(\pi, f) \in \mathcal{S}([n], X)$ with 
$f_{\pi}=(x_1, x_2, \dots, x_n)$,
let $\mathbf{inc}(f_\pi)=(x_{(1)}, x_{(2)}, \dots, x_{(n)})$ be the non-decreasing rearrangement of the terms of $f_{\pi}$. 
Set 
\[
i(f)=\max \{k: x_{(j)} \leq z_{j-1} \, \forall j \leq k \}.
\]
Thus, the maximality of $i=i(f)$ means that
\[
x_{(1)} \leq z_0, \,\,x_{(2)} \leq z_1,...\,,x_{(i)} \leq z_{i-1}
\]
 and 
\[
z_i < x_{(i+1)} \leq x_{(i+2)} \leq \cdots \leq x_{(n)} \leq x.\]
In the case that $x_{(j)} > z_{j-1}$ for all $j$, we have  $i(f)=0$.

Assume $(x_{r_1},...,x_{r_{i}})$ is the subsequence of 
$f_{\pi}$ from which the non-decreasing sequence $(x_{(1)}, x_{(2)},..., x_{(i)})$ is obtained. Let $R_1=\{ r_1, r_2, \dots, r_i\} \subseteq [n]$.  Then it is easy to see that $R_1$ must be 
a union of some blocks of $\pi$, while $R_2= [n]\setminus R_1$ is the union of the remaining blocks of $\pi$. 
Let $\pi_1$ be the restriction of $\pi$ on $R_1$ and $\pi_2$ the restriction of $\pi$ on $R_2$. Thus $\pi$ is a disjoint union of $\pi_1$ and $\pi_2$. 
Furthermore, let $f_i$ be the restriction of $f$ on $R_i$. Then 
$f_1$ is a map from the blocks of $\pi_1$ to $\{1, \dots, z_i\}$ that is also a $\vec{z}$-parking function, and $f_2$ is a map from blocks of $\pi_2$ to the set $X\setminus [z_{i}] = \{ 
z_{i}+1, \ldots, x\}$.  

Let $\mathcal{S}^P(E, X)$ be the subset of $\mathcal{S}(E, X)$ such that for each pair $(\pi, f)$, the sequence $f_\pi$ is a 
$\vec{z}$-parking function. Then the above argument defines a decomposition of $(\pi, f) \in \mathcal{S}([n], X)$ into pairs 
$(\pi_1, f_1) \in \mathcal{S}^P(R_1, X)$ and 
$(\pi_2, f_2) \in \mathcal{S}(R_2, X\setminus [z_{i}])$. 
Conversely, any pairs of $(\pi_1, f_1)$ and $(\pi_2, f_2)$ 
described above can be reassembled into a partition $\pi$ of $[n]$ with labels in $X$. In other words, the set $\mathcal{S}([n], X)$ can be written as a disjoint union of Cartesian products  as 
\begin{equation} \label{eq:decomposition}
\mathcal{S}([n], X) = 
\bigcupdot_{i; R_1 \subseteq [n]:  |R_1|=i} 
\ \mathcal{S}^P(R_1, X) \times \mathcal{S}(R_2, X\setminus [z_{i}]).
\end{equation}
In addition, if $\pi$ is the disjoint union of $\pi_1$ and $\pi_2$, then 
\[
w(\hat 0, \pi) = w(\hat 0, \pi_1) w(\hat 0, \pi_2).
\]
Putting the above results together, we have 
\begin{eqnarray*} 
a_n( x;w) &= & \sum_{(\pi, f) \in \mathcal{S}([n],X) } w(\hat 0, \pi) \\ 
& =  & \sum_{i=0}^n \sum_{R_1: |R_1|=i} \left( 
\sum_{(\pi_1, f_1) \in \mathcal{S}^p(R_1, X)} w(\hat 0, \pi_1) 
\cdot 
\sum_{(\pi_2, f_2) \in \mathcal{S}(R_2, X\setminus [z_{i}]) } w(\hat 0, \pi_2) \right) \\ 
& = & \sum_{i=0}^n \binom{n}{i} a_{n-i}( x-z_{i};w) 
\sum_{(\pi_1, f_1) \in \mathcal{S}^p(R_1, X)} w(\hat 0, \pi_1) \\
 & = & \sum_{i=0}^n \binom{n}{i} a_{n-i}(x-z_{i};w) 
 \sum_{\pi \in \Pi_i} w(\hat 0, \pi) PF_{\pi}(\zz) .
\end{eqnarray*} 
The last equation follows from the definition of $PF_{\pi}(\zz)$. 
\end{proof}

\begin{example} 
From the recurrence \eqref{recurrdef} we get 
\[
t_2(x;w, \zz)= x^2+ (w_2-2z_1)x+ (2z_0z_1-z_0^2-w_2z_0).
\]
Hence $t_2(0; w, -\zz) = 2z_0z_1-z_0^2+w_2 z_0$.  
On the other hand, there are two partitions in $\Pi_2$. 
For $\pi=\{12\}$, clearly $w(\hat 0, \{12\})=w_2$ and $PF_{ \{12\}}( \zz)=z_0$. For $\pi=\{1\}\{2\}$, $w(\hat 0, \pi)=1$ 
and $PF_\pi(\zz)$ is the number of pairs of positive integers $(x,y)$ such that $\min(x,y) \leq z_0$ and $\max(x,y) \leq z_1$. It is easy to check that there are $2z_0z_1-z_0^2$ such pairs. 
\qed
\end{example}

Since $\{ a_n(x;w)\}$ gives a generic form of the sequence of polynomials of binomial type, $\{t_n(x; w, \zz)\}$ 
is the generic form of the generalized \gon polynomials. 
In particular,  from Theorem~\ref{thm:gp-partition-lattice} 
we see that when $w_2=w_3=\cdots =0$, $t_n(0; w, -\zz)$
gives the number of $\vec{z}$-parking functions of length $n$.

\section{\gon Polynomials in Exponential Families}


In this section, we explore a more concrete realization of  generalized \gon polynomials in exponential families, which are  picturesque models that deal with counting structures that are built out of connected pieces and can be applied to many combinatorial  problems.

\subsection{Exponential Families}
Exponential families are  combinatorial models based on the partition lattices where the enumeration are captured by the exponential generating functions. The description of exponential families and their relation to the incidence algebra of $\Pi_n$ can be found in standard textbooks, 
e.g., \cite[Section 5.1]{EC2}. 
Here we adopt  Wilf's description of exponential families \cite{wilf} in the context of `playing cards' and `hands'.

Suppose that there is given an abstract set $P$ of `pictures', which typically are the connected structures. A \emph{card}
$\mathcal{C}(S, p)$ is a pair consisting of a finite label set 
$S$ of positive integers and a picture $p \in P$. The weight of
$\mathcal{C}$ is $|S|$. If $S=[n]$, the card is called \emph{standard}. 
A \emph{hand} $H$  is a set of cards whose label sets form a partition of $[n]$ for some $n$. The weight of a hand is the sum of the weights of the cards in the hand.  
The \emph{$n$-th deck} $\mathcal{D}_n$ is the set of all standard cards of weights $n$. We require that  $\mathcal{D}_n$ is always finite. 
An \emph{exponential family} $\mathcal{F}$ is the  collection of decks $\mathcal{D}_1, \mathcal{D}_2, \dots $. 

In an exponential family, let $d_i=|\mathcal{D}_i|$ and $h_{n,k}$ be the number of hands $H$ of weight $n$ that consist of $k$ cards. Let $h_0(x)=1$ and for $n \geq 1$, 
\begin{equation}‎
h_{n}(x)= \sum_{k = 1}^n h_{n,k} x^k. 
\end{equation}
 Then the main counting theorem, \emph{the exponential formula}, 
states that  these polynomials satisfy the generating relation: 
\begin{equation} \label{eq:h-gf}
e^{x{D}(t)}= ‎‎\sum\limits_{n\geq 0} h_{n}(x) \frac{t^n}{n!},
\end{equation}
where $D(t)=\sum_{k \geq 1} d_k t^k/k!$. 
In other words, if $d_1=h_{1,1}\neq 0$, $\{h_n(x)\}_{n \geq 0}$ is a sequence of binomial type that is conjugate to the delta operator 
$\sum_{k \geq 1} d_k D^k/k!$.

\begin{example} \label{example4-partition}
 \textit{Set Partitions}:
Here, a card is a label set $[n]$ with a `picture' of $n$ dots. Each deck $\mathcal{D}_n$ consists of the single card of weight $n$, and a hand  is just a partition of the set $[n]$. Thus, $h_{n,k}$ is the  number of partitions of the set $[n]$ into $k$ classes, which is $S(n,k)$, the Stirling number of the second kind.   \qed 
\end{example}

\begin{example} \label{example5-cycles} 
 \textit{Permutations and their Cycles}:
Each card is a cyclic permutation on a label set $S$. 
The deck $\mathcal{D}_n$ consists of all distinct cyclic permutations on $[n]$ so $d_n=(n-1)!$ and a hand is a permutation of $[n]$ consisting of $k$ cycles. Thus, $h_{n,k}$ is the number of permutations on $[n]$ that have $k$ cycles, that is, the signless Stirling number of the first kind 
$c(n,k)$. \qed 
\end{example} 

Note that we can interpret  $x^k$ in $h_n(x)$ as the 
number of maps from the set of cards in a hand to the set 
$X=\{1,\dots, x\}$ for some positive integer $x$. 
Hence $h_n(x)$ counts the number of hands of weight $n$ in which each card is labeled by an element of $X$. 
This set-up gives 
a natural combinatorial interpretation for  binomial polynomial sequences whose coefficients are positive integers. 

In \emph{Foundation III}  \cite{mullinrota} Mullin and Rota 
introduced a structure called \emph{reluctant functions}, 
which can be used to give a combinatorial interpretation for some generalized \gon polynomials; see \cite{goncarovdelta}. We remark that reluctant 
functions is a special case of exponential families, in which 
the  `pictures' are certain sets of trees. 
We will show 
that by taking the type enumerator an exponential family actually provides a combinatorial model for all generalized \gon polynomials. 

\vanish{
\begin{remark}
Let $S$  and $X$ be finite disjoint sets. We recall from \cite{mullinrota} that a reluctant function from $S$ to $X$  is a function $f:S\to S\cup X$, such that for every $s \in S$ there is a positive integer $k=k(s)$ with $f^k(s) \in X$. The resulting partition that arises from this construction has a natural combinatorial structure of a rooted forest with its rooted trees belonging to some binomial class $B$. The concept of exponential families generalizes reluctant functions since for each such binomial class $B$, a card consists of all $s\in S$ such that $f^k(s)=x_i$ $(k>0)$ for some fixed $x_i \in X$. The picture on a card is the underlying tree structure of $B$. The deck $D_n$ consists all trees on $n$ nodes of the particular family $B$, and a hand is a forest of rooted trees, where each such tree is in $B$. Thus, $h(n,k)$ is the number of forests on $n$ vertices consisting of $k$ rooted trees, each of type $B$.  
\end{remark}
}

\subsection{Type Enumerator in  Exponential Families}

In a given exponential family $\mathcal{F}$, 
we have seen that 
\begin{equation} \label{eq-h(x)}
h_n(x)= \sum_{k=1}^n h_{n,k}x^k= \sum_{H} | \{ f: \text{ cards in $H$ } \rightarrow X \} | 
\end{equation} 
where $H$ ranges over all hands of weight $n$.  
For a hand $H$ consisting of cards $\mathcal{C}_1$, 
$\mathcal{C}_2, \dots, \mathcal{C}_k$ of weights $t_1, t_2, \dots, t_k$, define the type of $H$ as 
\[
\mathrm{type}(H) = y_{t_1} y_{t_2} \cdots y_{t_k}, 
\]
where $y_1, y_2, \dots, $ are free variables. 

Let 
\begin{equation} \label{eq:h-type} 
h_n(x;\mathbf{y}) = \sum_{H: \text{ weight } n } \mathrm{type}(H) | \{ f: \text{ cards in $H$ } \rightarrow X \} |. 
\end{equation} 
Then we have the following form of the exponential formula.


\begin{prop}\label{prop:type-exponential}
The sequence of type enumerators  $\{ h_n(x;\mathbf{y}) \}_{n\geq 0}$, 
viewed as a polynomial in $x$,  is a sequence of polynomials of binomial type satisfying the equation 
\begin{equation} \label{eq:h-type-gf} 
\sum_{n \geq 0} h_n(x;\mathbf{y}) \frac{t^n}{n!} 
= \exp\left(x \sum_{k \geq 1} d_k y_k \frac{t^k}{k!} \right) .
\end{equation} 
\end{prop}
\begin{proof} 
We compare the formula of $h_n(x; \mathbf{y})$ with that of $h_n(x)$. Note for $n \geq 1$, $h_n(x)$ can be computed by
\begin{equation}
h_n(x) = \sum_{k\geq 1} \sum_{H=\{ \mathcal{C}_1, \dots, \mathcal{C}_k\} }d_{t_1} d_{t_2} \cdots d_{t_k} x^k,
\end{equation}
where $\{\mathcal{C}_1, \cdots, \mathcal{C}_k\}$ is a hand of weight $n$ and 
$t_i$ is the weight of card $\mathcal{C}_i$, while 
\begin{equation}
h_n(x;\mathbf{y}) = \sum_{k\geq 1} \sum_{H=\{ \mathcal{C}_1, \dots, \mathcal{C}_k\} } d_{t_1} d_{t_2} \cdots d_{t_k} 
y_{t_1}y_{t_2} \cdots y_{t_k}  x^k. 
\end{equation}
The exponential formula for $h_n(x)$ then implies Proposition~\ref{prop:type-exponential}.  
\end{proof} 

\begin{remark} 
Comparing to the generic form $a_n( x;w)$ in the previous section, we see that $h_n(x;\mathbf{y})$ corresponds to the case where the variables in the zeta-type function are determined by  
$w_n=d_ny_n$. As far as $d_1\neq 0$, we can obtain arbitrary
polynomial  sequences of binomial type by taking  suitable values for the $y_i$-variables.

\end{remark}

\subsection{Sequence of Generalized Goncarov Polynomials} 
Let $Z=(z_i)_{i\geq 0}$ be an interpolation grid.  For the 
binomial sequence $\{h_n(x;\mathbf{y})\}_{n \geq 0}$ defined in an 
exponential family $\mathcal{F}$, we can consider the 
associated generalized \gon polynomials given by \eqref{eq:gp-recurrence} with $p_n(x)$ replaced by $h_n(x;\mathbf{y})$. 
Denote this \gon polynomial by $t_n(x; \mathbf{y}, \mathcal{F}, \zz)$ to emphasize that it has variables $y_i$ and is defined in $\mathcal{F}$. Explicitly,  
 $t_n(x; \mathbf{y}, \mathcal{F}, \zz)$ is obtained by the recurrence 
\begin{equation}\label{eq:4}
t_{n}(x;\mathbf{y},\mathcal{F}, \zz)= h_n(x;y) - \sum\limits_{i=0}^{n-1} \binom{n}{i} h_{n-i}(z_i;y) t_i(x;\mathbf{y}, \mathcal{F}, \zz).
\end{equation}

Suppose $X=\{1,2,...,x\}$ and assume that $z_0 \leq z_1 \leq \cdots \leq z_{n-1}$ are integers in $X$.
Let $\vec{z}=(z_0, z_1, \dots, z_{n-1})$. 
For a hand $H=\{ \mathcal{C}_1, \mathcal{C}_2, \dots, \mathcal{C}_k\}$ of weight $n$ with a function $f$ from 
$\{\mathcal{C}_1, \mathcal{C}_2, \dots, \mathcal{C}_k\}$ to  $X$, 
denote by $f_H$ the list $(x_1, x_2, \dots, x_n)$, where 
$x_i = f(\mathcal{C}_j)$ if $i$ is in the label set of $\mathcal{C}_j$. 
Let 
\[
\mathcal{PF}_H(\zz) = \{ f: \{\mathcal{C}_1, \mathcal{C}_2, \dots, \mathcal{C}_k\} \rightarrow X\ | \ f_H \text{ is a 
$\vec{z}$-parking function} \},
\]
and $PF_H(\zz)$ the cardinality of $\mathcal{PF}_H(\zz)$. 
Then we have the following analog of Theorem~\ref{thm:gp-partition-lattice}. 

\begin{theorem}\label{thm:gp-exponential}
For $n \geq 0$, 
\begin{equation}\label{eq:main2} 
 t_n(0; \mathbf{y}, \mathcal{F}, -\zz) =t_n(x; \mathbf{y}, \mathcal{F},x-\zz)= \sum_{H: \text{ of weight } n} \mathrm{type}(H) \cdot 
 PF_{H}(\zz).  
 \end{equation}
 Here by convention, the third term of \eqref{eq:main2}  equals  $1$ when $n=0$. 
 \end{theorem}

Theorem~\ref{thm:gp-exponential}  follows from \eqref{eq:4} 
and the following recurrence relation 
\begin{equation}
\label{eq:mainlemma2}
   h_n(x;\mathbf{y})= \sum_{i=0}^n \binom{n}{i} h_{n-i}(x-z_i;\mathbf{y}) 
      \sum_{H: \text{ of weight } i}  \mathrm{type}(H) \cdot PF_{H}(\zz), 
\end{equation}
whose proof is similar to that of Lemma \ref{mainlem}.
In an exponential family $\mathcal{F}$, let 
$\mathcal{A}(S, X)$ be the set of pairs 
$(H, f)$ such that  $H$ is a hand whose label sets form a partition of $S$ and  $f$ is a function from the cards in $H$
to  $X$. 
Then the basic ingredients of the proof are that 
\begin{enumerate}[(1)] 
\item  $\mathrm{type}(H)$ is a multiplicative function only depending  on the weights of cards in $H$, and  
\item  The set $\mathcal{A}([n], X)$ 
 can be decomposed into a disjoint union of Cartesian products of the form 
 \[
 \mathcal{A}^P(R, X) \times \mathcal{A}([n]\setminus R, X\setminus [z_i]), 
 \]
 where $\mathcal{A}^P(R,X)=\{ (H, f) \in \mathcal{A}(R, X): \ 
 f_H \text{ is a $\vec{z}$-parking function} \}$, and the disjoint union is taken over all the subsets $R$ of $[n]$. 
 \end{enumerate} 
 We skip the details of the proof of Eq. \eqref{eq:mainlemma2}. 
 
We illustrate the above results and some connections to combinatorics  in the exponential families given in Examples \ref{example4-partition} and \ref{example5-cycles}. There are many other  exponential families in which the type enumerator and associated \gon polynomials have interesting combinatorial significance. 

\begin{enumerate} 
\item Let $\mathcal{F}_1$ be the exponential family of set partitions described in Example \ref{example4-partition}. In this family, $d_i=1$ for all $i$ and $h_n(x)= 
\sum\limits_{k=0}^n S(n,k) x^k$. In the type enumerator, if we substitute 
 $y_1=1$ and $y_i=w_i$ for $i \geq 2$, then $h_n(x;\mathbf{y})$ is exactly the same as the generic sequence $a_n(x;w)$ in \eqref{eq:zeta-polynomial}, and consequently $t_n(x; \mathbf{y}, \mathcal{F}_1, \zz)$ is the same as the generic \gon polynomial 
 $t_n(x;w,\zz)$ defined by \eqref{recurrdef2}. 
In particular, if all $y_i=1$, $t_n(0; \mathbf{y}, \mathcal{F}_1, -\zz)$ gives a formula for the number of $\vec{z}$-parking functions  
with the additional structure that cars arrive in disjoint groups, and drivers in the same group always prefer the same parking spot. 

When $y_i=1$ and $z_i=1+i$ for all $i$, the first few terms of the \gon polynomials $t_n(x)=t_n(x;\mathbf{y}, \mathcal{F}_1, -\zz)$ are
 \begin{eqnarray*} 
 t_0(x) &= & 1   \\
 t_1(x) & =& x+1   \\
 t_2(x) & =& x^2+5x+4  \\
 t_3(x) & = & x^3+12x^2+40x+29 \\
 t_4(x) &= & x^4+ 22x^3+163x^2+ 453x+ 311  
 \end{eqnarray*} 
In particular, for $x=0$  we get the sequence $1,1,4,29,311,...$. This is sequence  A030019 in the On-Line Encyclopedia of Integer Sequences (OEIS) \cite{oeis}, where it is interpreted as the number of labeled spanning trees in the complete hypergraph on $n$ vertices (all hyper-edges having cardinality 2 or greater). It would be interesting to find a direct bijection between the hyper-trees and the 
parking-function interpretation.

\item Let $\mathcal{F}_2$ be the exponential family of the permutations and their cycles, as described in Example \ref{example5-cycles}. Here $d_n=(n-1)!$ and 
 $h_n(x)=\sum\limits_{k=0}^{n} c(n,k)x^k=x^{(n)}$, where the $c(n,k)$ is the signless Stirling numbers of the first kind and $x^{(n)}$ is the rising factorial $x(x+1)\cdots (x+n-1)$.  When $y_1=1$, the \gon  polynomial 
  $t_n(x; \mathbf{y}, \mathcal{F}_2, \zz)$ can be obtained from 
  the generic form $t_n(x; w, \zz)$ by replacing $w_n$ with 
  $(n-1)!y_n$ for $n \geq 2$.  
  When all $y_i=1$, i.e. $y =\mathbf{1}$,
 $t_n(0; \mathbf{1}, \mathcal{F}_2, -\mathcal{Z})$ gives a formula for the number of $\vec{z}$-parking functions 
with the addition requirement that  cars are formed in disjoint cycles, and drivers in the same cycle prefer the 
same parking spot. 
  
  In addition, when $y =\mathbf{1}$, and 
 $\mathcal{Z}$ is the arithmetic progression $z_i=a+bi$, the Goncarov polynomial  is  
\begin{equation} \label{eq:family2} 
t_n(x; \mathbf{1}, \mathcal{F}_2; -\mathcal{Z})=(x+a)
(x+a+nb+1)^{(n-1)}. 
\end{equation}

Another combinatorial interpretation of  $t_n(0; \mathbf{1}, \mathcal{F}_2, -\zz)$ 
is  given in \cite[Section 6.7]{goncarovdelta}, where it shows  that  $t_n(0;\mathbf{1}, \mathcal{F}, -\zz)$  is $n!$ times the number of lattice paths from $(0,0)$ to $(x-1,n)$ with strict  right boundary $\zz$. 
 For example, when $z_i=a+bi$ for some  positive 
 integers $a$ and  $b$, $\frac{1}{n!} t_n(0; \mathbf{1}, \mathcal{F}_2, -\zz)$ is the number of lattice paths from $(0,0)$ to $(x-1, n)$ which stay strictly to the left of the points 
 $(a+ib, i)$ for $i=0, 1, \dots, n$. 
 In particular for $a=1$ and $b=k$,  it counts the number of labeled lattice paths from the origin to $(kn,n)$ that never pass below the line $x=yk$. In that case 
 \eqref{eq:family2}  gives
 $\frac{1}{1+kn}\binom{(k+1)n}{n}$,  the $n$-th $k$-Fuss-Catalan number.

\end{enumerate}

\begin{remark}
 We can also consider the injective functions  in the  definition 
 of $h_n(x)$ and $h_n(x;y)$ in \eqref{eq-h(x)} and \eqref{eq:h-type}, where the term $x^k$ is replaced by the lower factorial $x_{(k)}=x(x-1)\cdots (x-k+1)$.  In other words,  cards of a hand  are labeled by $X$ with the additional property that different cards get different labels.
Some examples are given in \cite[Section 6]{goncarovdelta} and 
called \emph{monomorphic classes}.  
 A result analogous to Theorem~\ref{thm:gp-exponential}  still holds for the momomorphic classes of an exponential family. 
 \end{remark}

As a final result we point out an explicit formula  to compute
the constant coefficient of the generalized \gon polynomial whenever we know the basic sequence $\{p_n(x)\}_{n \geq 0}$. 
It is proved in \cite{goncarovdelta} and only depends on the recurrence \eqref{eq:gp-recurrence} and the fact that $p_n(0)=0$ for $n > 0$.  The proof  
does not need an explicit formula for the delta operator 
$\Delta$ and hence the result is easier to use when we need to compute the value of $t_n(0; \mathbf{y}, \mathcal{F},-\zz)$ in a given  exponential family. 

Let $\{ p_n(x)\}_{n\geq 0}$ be a sequence of binomial type 
and $\zz=(z_0, z_1, \dots  )$  be a given  grid.  
Assume $\{t_n(0;-\zz)\}_{n \geq 0}$ is defined by the recurrence relation 
\[
t_n(0;-\zz) = -\sum_{i=0}^{n-1} \binom{n}{i} p_{n-i}(-z_i) t_i(0), 
\]
for $n \geq 1$ and $t_0(0;-\zz)=1$. 
Then for $n \geq 1$, $t_n(0;-\zz)$ can be expressed as a summation 
over ordered partitions.

Given a finite set $S$ with $n$ elements, an \textit{ordered partition} of $S$ is an ordered list $(B_1,...,B_k)$ of disjoint nonempty subsets of $S$ such that $B_1\cup \cdots \cup B_k=S$.
If $\rho =(B_1,...,B_k)$ is an ordered partition of $S$, then we set $|\rho|=k$.
For each $i=1,2,...,k$, we let $b_i=b_i(\rho)=|B_i|$, and $s_i:=s_i(\rho):=\sum_{j=1}^{i} b_j$. In particular,  set $s_0(p)=0$.
Let $\mathcal{R}_n$ be the set of all ordered partitions of the set $[n]$.
\begin{theorem}[\cite{goncarovdelta}] \label{thm:gp-form} 
For $n \geq 1$, 
\begin{align}\label{4.7}
    t_n(0;-\zz) =& \sum\limits_{\rho \in \mathcal{R}_n} (-1)^{|\rho|} \prod_{i=0}^{k-1} p_{b_{i+1}}(-z_{s_i}) \nonumber \\ 
    =& \sum\limits_{\rho \in \mathcal{R}_n} (-1)^{|\rho|} p_{b_1}(-z_{0}), \cdots \,p_{b_{k}}(-z_{s_{k-1}}).
\end{align}
\end{theorem}
The following list is the formulas for the first several \gon polynomials. 
\begin{eqnarray*} 
t_0(0;-\zz) &= & 1 \\
t_1(0;-\zz) & = & -p_1(-z_0) \\ 
t_2(0; -\zz) & = & 2p_1(-z_0) p_1(-z_1) -p_2(-z_0) \\ 
t_3(0; -\zz) & =& -p_3(-z_0) +3p_2(-z_0) p_1(-z_2) +3 p_1(-z_0)p_2(-z_1) -6 p_1(-z_0)p_1(-z_1)p_1(-z_2).
\end{eqnarray*}

\subsection{Degenerate Cases}

In an exponential family, the polynomial $h_n(x)$ or $h_n(x;\mathbf{y})$ may not always have degree $n$, e.g., when $d_1=h_{1,1}=0$. We say that such polynomial sequences and the corresponding  exponential families are \emph{degenerate}. 
For a degenerate sequence of polynomials, there is no delta operator for which the sequence is the basic or the conjugate sequence. 
Nevertheless, the 
exponential formulas \eqref{eq:h-gf} and \eqref{eq:h-type-gf}
are still true. Hence the sequences 
$\{ h_n(x)\}_{n \geq 0}$ and $\{ h_n(x;\mathbf{y})\}_{n\geq 0}$ still satisfy the binomial-type identity \eqref{eq:binomial}. 

Without a delta operator, we cannot define 
the generalized \gon interpolation problems. However, we can still introduce the generalized \gon polynomials via the recurrence \eqref{eq:gp-recurrence}. 
Furthermore, we will prove in Theorem \ref{thm:shift} that the shift invariance of \gon polynomials can also be derived from \eqref{eq:gp-recurrence}. Therefore, 
Theorems \ref{thm:gp-exponential} and \ref{thm:gp-form} still hold true for the degenerate exponential families since all the proofs follow from the binomial-type  identity 
\eqref{eq:binomial} 
and the recurrence \eqref{eq:gp-recurrence}. 

\begin{theorem} \label{thm:shift} 
Assume $\{p_n(x)\}_{n\geq 0}$ is a polynomial sequence of binomial type
with $p_0(x)=1$, but the degree of $a_n(x)$ is not necessary $n$. Let $t_n(x;\zz)$ be defined by the recurrence relation 
\begin{eqnarray}
 t_n(x;\zz) =p_n(x)-\sum_{i=0}^{n-1} \binom{n}{i} p_{n-i}(z_i) t_i(x;\zz). 
\end{eqnarray}
For any scalar $\eta$ and the interpolation grid $\zz=\{z_0, z_1, z_2, \dots)$, let $\zz+\eta$ be the sequence $(z_0+\eta, z_1+\eta,z_2+\eta,  \cdots)$.  Then we have 
\begin{equation}\label{eq:shift-invt} 
 t_n(x+\eta; \zz+\eta) = t_n(x;\zz)
\end{equation}
for all $n \geq 0$. 
\end{theorem}
\begin{proof}
We prove Theorem \ref{thm:shift} by induction on $n$. The initial case $n=0$ is trivial since $t_0(x;\zz)=1$ for all $x$ and any grid $\zz$. Assume Eq.~\eqref{eq:shift-invt} is true for all indices less than $n$. We compute 
$t_n(x+\eta; \zz+\eta)$. By definition 
\begin{eqnarray} \label{eq:middle}
t_n(x+\eta;\zz+\eta) = p_n(x+\eta) -\sum_{i=0}^{n-1} 
\binom{n}{i} p_{n-i}(z_i+\eta) t_i(x+\eta;\zz+\eta). 
\end{eqnarray}
By the inductive hypothesis $t_i(x+\eta;\zz+\eta)=t_i(x;\zz)$ for $i<n$ and the binomial identity of $p_n(x)$, the right-hand side of 
\eqref{eq:middle} can be written as 
\begin{eqnarray}
& & \sum_{k=0}^n \binom{n}{k} p_k(x)p_{n-k}(\eta) 
- \sum_{i=0}^{n-1} \binom{n}{i} \left( \sum_{j=0}^{n-i} \binom{n-i}{j} p_{n-i-j}(z_i) p_j(\eta)  \right) t_i(x;\zz) \nonumber  \\
& = &  \sum_{k=0}^n \binom{n}{k} p_k(x)p_{n-k}(\eta) -\sum_{i+j\leq n \atop \text{except  } (i,j)=(n,0)} \binom{n}{i} \binom{n-i}{j} p_j(\eta) p_{n-i-j}(z_i) t_i(x;\zz)  \label{step 1} 
\end{eqnarray}
Since
\[
\binom{n}{i} \binom{n-i}{j} = \frac{n!} {i! j! (n-i-j)!} = \binom{n}{j} \binom{n-j}{i},  
\]
then~\eqref{step 1} can be expressed as 
\begin{eqnarray} 
& & \sum_{k=0}^n \binom{n}{k} p_k(x)p_{n-k}(\eta) -\sum_{i+j\leq n \atop \text{except  } (i,j)=(n,0)} \binom{n}{j} \binom{n-j}{i} p_j(\eta) p_{n-i-j}(z_i) t_i(x;\zz) \nonumber \\ 
&= & \sum_{k=0}^n \binom{n}{k} p_k(x)p_{n-k}(\eta) 
 -\sum_{j=1}^{n} \binom{n}{j} p_j(\eta) \sum_{i=0}^{n-j} \binom{n-j}{i} p_{n-i-j}(z_i) t_i(x;\zz) \nonumber \\
& &  - \sum_{i=0}^{n-1} \binom{n}{i} p_{n-i}(z_i) t_i(x;\zz). 
\label{eq:shift2} 
\end{eqnarray}
The last summation in \eqref{eq:shift2} corresponds to the terms with $j=0$.  
Note that 
\[
 \sum_{i=0}^{n-j} \binom{n-j}{i} p_{n-i-j}(z_i) t_i(x;\zz)
 =p_{n-j}(x). 
 \]
 Hence For.~\eqref{eq:shift2} is equal to 
 \begin{eqnarray*}
& &  \sum_{k=0}^n \binom{n}{k} p_k(x)p_{n-k}(\eta) -\sum_{j=1}^n \binom{n}{j} p_j(\eta) p_{n-j}(x) 
  - \sum_{i=0}^{n-1} \binom{n}{i} p_{n-i}(z_i)t_i(x;\zz)  \\
 &  =&  p_n(x) - \sum_{i=0}^{n-1} \binom{n}{i} p_{n-i}(z_i)t_i(x;\zz) \\
  & = & t_n(x;\zz). 
 \end{eqnarray*}
This finishes the proof. 
\end{proof}

The next example shows a degenerate exponential family. 

\begin{example} 
\textit{2-Regular simple graphs}. 
In this exponential family a card is an undirected cycle on a label set $[m]$ (where $m \geq 3$). The deck $\mathcal{D}_n$ consists of all undirected circular arrangements of $n$ letters so $d_n=\frac{1}{2}(n-1)!$ for $n  \geq 3$ and $d_1=d_2=0$.  A  hand is then a undirected simple graph on the vertex set $[n]$, 
which is 2-regular, that is, every vertex has degree 2.  
Thus, $h_{n,k}$ is the  number of  undirected 2-regular simple graphs on $n$ vertices consisting of $k$ cycles. 
Denote by $\mathcal{F}_3$ this exponential family. 

For $\mathcal{F}_3$, the type enumerators are
$h_0(x,\mathbf{y})=1$, 
 $h_1(x;\mathbf{y})=
 h_2(x;\mathbf{y})=0$,  
 $h_3(x;\mathbf{y})= y_3x$, 
 $h_4(x;\mathbf{y})=2y_4 x$, 
 $h_5(x,\mathbf{y})=12 y_5 x$, and $ 
 h_6(x;\mathbf{y})= 60y_6 x+ 10 y_3^2 x^2$, etc. 
 Although the degree of $h_n(x;\mathbf{y})$ is not $n$, 
the exponential formula still holds: 
\[
\sum_{k=0}^n h_n(x;\mathbf{y}) \frac{t^k}{k!} = \exp
\left(  x \sum_{k \geq 3 } y_k \frac{t^k}{2k} \right). 
\] 
We compute by the recurrence \eqref{eq:4}  that 
\begin{eqnarray*} 
t_0(x; \mathbf{y}, \mathcal{F}_3, \zz) &= &1 \\ 
t_1(x; \mathbf{y}, \mathcal{F}_3, \zz) & =&  t_2(x;y, \mathcal{F}_3, \zz) =
 0 \\
 t_3(x; \mathbf{y}, \mathcal{F}_3, \zz) &=& y_3(x-z_0), \\
 t_4(x;\mathbf{y}, \mathcal{F}_3, \zz) &= & 3y_3 (x-z_0), \\
 t_5(x; \mathbf{y}, \mathcal{F}_3, \zz) & =& 12y_5 (x-z_0), \\
 t_6(x; \mathbf{y}, \mathcal{F}_3, \zz) &= &  10y_3^2 x^2 +60y_6x-20y_3^2 z_3x -60y_6z_0-10y_3^2z_0^2 +20y_3^2 z_0z_3. 
\end{eqnarray*} 
The equation 
\[
t_n(0; \mathbf{y}, \mathcal{F}_3, -\zz) = \sum_{H: \text{ of weight $n$}} \mathrm{type}(H)\cdot PF_H(\zz) 
\]
is still true. For example, for $n=6$, 
$t_6(0; y, \mathcal{F}_3, -\zz)= 60y_6 z_0+ 20y_3^2 z_0z_3-10y_3^2z_0^2$.  The term $60 y_6 z_0$ comes from 
the $5!/2=60$ $6$-cycles, and the terms $10y_3^2 (2z_0z_3-z_0^2)$
comes from the 10 hands each with two $3$-cycles. 
 \qed

\end{example}



\section{Closing Remarks}

In this paper we present the combinatorial interpretation of an arbitrary sequence of \gon polynomials associated with a polynomial sequence of binomial type. There are many other combinatorial problems that provide a formal framework of coalgebras,  bialgebras, or Hopf algebras \cite{jonirota}. 
In those problems the counting sequences satisfy an identity that is analogous to the binomial-type identity \eqref{eq:binomial}, 
with the binomial coefficients $\binom{n}{i}$ replaced by some other section coefficients.  
For example, the theory of binomial enumeration proposed by  Mullin and Rota \cite{mullinrota} was generalized to an abstract context and applied to dissecting schemes by Henle \cite{henle}.  It would be an interesting  project to investigate the role of generalized \gon polynomials in these other dissecting schemes and discrete structures.  
As suggested by Henle, this research may lead to 
connections to rook polynomials, order invariants of posets, Tutte invariants of combinatorial geometries, 
cycle indices and symmetric functions, and many others. 




\end{document}